\documentclass[10pt,reqno]{amsart} 
\usepackage{amssymb}
\usepackage[cp1250]{inputenc}   
\usepackage{hyperref}
\usepackage[all]{xy}            
\usepackage{verbatim}
\usepackage{tikz}
    \usetikzlibrary{shadows,matrix,arrows}
\usepackage{mathtools}
\usepackage{longtable}
\usepackage{enumitem}

\newtheorem{Thm}{Theorem}[section]

\newtheorem{Lmm}[Thm]{Lemma}

\theoremstyle{definition}

\newtheorem{Exp}[Thm]{Example}
\theoremstyle{remark}

\numberwithin{equation}{section}

\def\Ker{\mathrm{Ker}}
\def\Im{\mathrm{Im}}
\def\N{\mathbb{N}}
\def\Z{\mathbb{Z}}
\def\Q{\mathbb{Q}}

\newcommand{\icases}[7]{#7\{\!\begin{smallmatrix}
                                #5#1\hfill;  & \,#5#2\hfill &\!\!\\[#6]
                                #5#3\hfill;  & \,#5#4\hfill &\!\!
                              \end{smallmatrix}}

\def\cone{\mathop{\rm Cone}\nolimits}

\newcommand{\ena}[1]{[#1]}          
\newcommand{\dva}[1]{[\![#1]\!]}    

\begin{document}
\title[Torsion table for the Lie algebra $\frak{nil}_n$]{Torsion table for the Lie algebra $\frak{nil}_n$}
\author{Leon Lampret}
\address{Department of Mathematics, University of Ljubljana, Slovenia}
\email{leon.lampret@fmf.uni-lj.si}
\author{Aleš Vavpetič}
\address{Department of Mathematics, University of Ljubljana, Slovenia}
\email{ales.vavpetic@fmf.uni-lj.si}
\date{November 2, 2017}
\keywords{algebraic/discrete Morse theory, homological algebra, chain complex, acyclic matching, nilpotent Lie algebra, triangular matrices, torsion table, algebraic combinatorics}
\subjclass[2010]{17B56, 13P20, 13D02, 55-04, 18G35, 58E05}

\begin{abstract}
We study the Lie ring $\mathfrak{nil}_n$ of all strictly upper-triangular $n\!\times\!n$ matrices with entries in $\Z$. Its complete homology for $n\!\leq\!8$ is computed.
\par We prove that every $p^m$-torsion appears in $H_\ast(\mathfrak{nil}_n;\Z)$ for $p^m\!\leq\!n\!-\!2$. For $m\!=\!1$, Dwyer proved that the bound is sharp, i.e. there is no $p$-torsion in $H_\ast(\mathfrak{nil}_n;\Z)$ when prime $p\!>\!n\!-\!2$. In general, for $m\!>\!1$ the bound is not sharp, as we show that there is $8$-torsion in $H_\ast(\mathfrak{nil}_8;\Z)$.
\par As a sideproduct, we derive the known result, that the ranks of the free part of $H_\ast(\mathfrak{nil}_n;\Z)$ are the Mahonian numbers (=number of permutations of $[n]$ with $k$ inversions), using a different approach than Kostant. Furthermore, we determine the algebra structure (cup products) of $H^\ast(\mathfrak{nil}_n;\Q)$.\vspace{-3mm}
\end{abstract}

\maketitle

\section{Introduction}
Let $\mathfrak{nil}_n$ be the Lie algebra of integral $n\!\times\!n$ strictly upper-triangular matrices. The complete homology $H_k(\mathfrak{nil}_n;\Z)$ is known only for $n\!\leq\!6$ \cite{citearticleJollenbeckADMTACA}:
\begin{center} \begin{tabular}{ l | l l l l l }
    $k\backslash n$ & 2 & 3 & 4 & 5 & 6 \\ \hline
    0 & $\Z$ & $\Z$ & $\Z$ & $\Z$ & $\Z$ \\
    1 & $\Z$ & $\Z^2$ & $\Z^3$ & $\Z^4$ & $\Z^5$ \\
    2 & & $\Z^2$ & $\Z^5\oplus\Z_2$ & $\Z^9\oplus\Z_2^2$ & $\Z^{14}\oplus\Z_2^{3}$ \\
    3 & & $\Z$ & $\Z^6\oplus\Z_2$ & $\Z^{15}\oplus\Z_2^8\oplus\Z_3^2$ & $\Z^{29}\oplus\Z_2^{20}\oplus\Z_3^4$ \\
    4 &&& $\Z^5$ & $\Z^{20}\oplus\Z_2^{10}\oplus\Z_3^3$ & $\Z^{49}\oplus\Z_2^{47}\oplus\Z_3^{13}\oplus\Z_4^{3}$ \\
    5 &&& $\Z^3$ & $\Z^{22}\oplus\Z_2^{10}\oplus\Z_3^3$ & $\Z^{71}\oplus\Z_2^{79}\oplus\Z_3^{26}\oplus\Z_4^{9}$ \\
    6 &&& $\Z$ & $\Z^{20}\oplus\Z_2^{8}\oplus\Z_3^2$ & $\Z^{90}\oplus\Z_2^{118}\oplus\Z_3^{35}\oplus\Z_4^{12}$ \\
    7 &&& & $\Z^{15}\oplus\Z_2^{2}$ & $\Z^{101}\oplus\Z_2^{138}\oplus\Z_3^{36}\oplus\Z_4^{12}$ \\
    8 &&& & $\Z^{9}$ & $\Z^{101}\oplus\Z_2^{118}\oplus\Z_3^{35}\oplus\Z_4^{12}$ \\
    9 &&& & $\Z^{4}$ & $\Z^{90}\oplus\Z_2^{79}\oplus\Z_3^{26}\oplus\Z_4^{9}$ \\
    10&&& & $\Z$ & $\Z^{71}\oplus\Z_2^{47}\oplus\Z_3^{13}\oplus\Z_4^{3}$ \\
    11&&& & & $\Z^{49}\oplus\Z_2^{20}\oplus\Z_3^{4}$\\
    12&&& & & $\Z^{29}\oplus\Z_2^{3}$\\
    13&&& & & $\Z^{14}$\\
    14&&& & & $\Z^5$\\
    15&&& & & $\Z$\\
\end{tabular}\end{center}
The main reason why computations for larger $n$ are exceedingly difficult is that the chain complex $C_\ast\!=\! \Lambda^{\!\ast}\mathfrak{nil}_n$ is immense. It has \smash{$2^{\binom{n}{2}}$} generators, which is more than 2 million for $n\!=\!7$. In the paper, we divide $C_\ast$ in numerous direct summands $\dva{w}$ (corresponding to sequences $w\!\in\!\{1,\ldots,n\}^n$ with $w_1\!+\!\ldots\!+\!w_n\!=\!\binom{n+1}{2}$) and show how many of them are isomorphic (up to dimension shift), many are contractible and many are obtained from smaller ones as the cone of a chain map $\dva{w'}\!\overset{\cdot t}{\to}\!\dva{w'}$. The direct summands corresponding to permutations of $(1,\ldots,n)$ are generated by just one element, hence those contribute only to free part. We show that any other direct summand contributes only to torsion, so we get $H^{\ast\!}(\frak{nil}_n;\Q)$.

\subsection*{Complex} Let $e_{ij}$ be the matrix with all entries 0 except 1 in the position $(i,j)$. The chain complex $C_\ast\!=\!\Lambda^{\!\ast}\mathfrak{nil}_n$, due to Chevalley (1948), is generated by wedges $e_{a_1b_1}\!\!\wedge\!\ldots\!\wedge\!e_{a_kb_k}$\!, where $1\!\leq\!a_i\!<\!b_i\!\leq\!n$ for all $i$. From now on, for the sake of brevity, we shall omit the $\wedge$ symbols. 
The boundary is defined by\vspace{-1mm}
\[\textstyle{\partial (e_{a_1b_1}\!\ldots e_{a_kb_k})= \sum_{i<j} (-1)^{i+j}
[e_{a_ib_i},e_{a_jb_j}]e_{a_1b_1}\!\ldots\widehat{e_{a_ib_i}}\!\ldots \widehat{e_{a_jb_j}}\!\ldots e_{a_kb_k},}\vspace{-1mm}\]
where $[e_{ab},e_{cd}]$ equals $e_{ad}$ if $b\!=\!c$, equals $-e_{cb}$ if $a\!=\!d$, and equals 0 otherwise.

\subsection*{AMT} For some computations later on, we shall use algebraic Morse theory, so we include a short review of it. To a chain complex of free modules $(C_\ast,\partial_\ast)$ we associate a weighted digraph $\Gamma_{C_\ast}$ (vertices are basis elements of $C_\ast$, weights of edges are nonzero entries of matrices $\partial_\ast$). Then we carefully select a matching $\mathcal{M}$ in this digraph, so that its edges have invertible weights and if we reverse the direction of every $e\!\in\!\mathcal{M}$ in $\Gamma_{B_\ast}$, the obtained digraph $\Gamma^{\mathcal{M}}_{C_\ast}$ contains no directed cycles and no infinite paths in two adjacent degrees. Under these conditions (i.e. if $\mathcal{M}$ is a \emph{Morse matching}), the AMT theorem  (\cite{citearticleSkoldbergMTFAV}, \cite{citearticleJollenbeckADMTACA}, \cite{citeKozlovDMTFCC}) provides a homotopy equivalent complex $(\mathring{C}_\ast,\mathring{\partial}_\ast)$, spanned by the unmatched vertices in $\Gamma^{\mathcal{M}}_{C_\ast}$, and with the boundary $\mathring{\partial}_\ast$ of $v\!\in\!\mathring{C}_k$ given by the sum of weights of directed paths in $\Gamma^{\mathcal{M}}_{C_\ast}$ to all critical $v'\!\in\!\mathring{C}_{k-1}$. For more details, we refer the reader to the three articles above (which specify the homotopy equivalence), or \cite{citeLampretVavpeticCLAAMT} for a quick introduction and formulation.

\vspace{3mm}
\section{Subcomplexes}
For a set $M\!=\!\{(a_1,b_1),\ldots,(a_k,b_k)\}\!\subseteq\! \{(i,j); 1\!\leq\!i\!<\!j\!\leq\!n\}$ we denote $e_M\!=\!e_{a_1b_1}\ldots e_{a_kb_k}$. For $M_i\!:=\!\{x; (i,x)\!\in\!M\}$ we have $e_M\!=\!\wedge_{i=1}^{\!n-\!1}\!e_{\{i\}\times M_i}$. We define the \emph{weight} vector $\widetilde{w}(e_M)\!=\!(\widetilde{w}_1,\ldots,\widetilde{w}_n)$ by $\widetilde{w}_i\!=\! |\{x;(x,i)\!\in\!M\}|\!-\!|\{y;(i,y)\!\in\!M\}|$, i.e. the number of times $i$ appears on the right in $e_M$ minus the number of times $i$ appears on the left in $e_M$. Then $\sum_{i=1}^n\!\widetilde w_i\!=\!0$. Every summand in $\partial(e_M)$ has the same weight as $e_M$. Therefore a submodule $\ena{\widetilde w}$ of $\Lambda^{\!\ast}\mathfrak{nil}_n$, spanned by the basis elements with weight $\widetilde{w}$, forms a chain subcomplex which is a direct summand.
\par Most equalities will be described more conveniently using the \emph{modified weight} $w(e_M)\!=\!(1,\ldots,n)\!-\!\widetilde{w}(e_M)\!=\!(1\!-\!\widetilde{w}_1,\ldots, n\!-\!\widetilde{w}_n)$. Then $\sum_{i=1}^n\!w_i\!=\!\binom{n+1}{2}$ and $i\!-\!n\!\leq\! \widetilde{w}_i\!\leq\! i\!-\!1$ implies $1\!\le\!w_i\!\leq\!n$ for all $i$. We denote $\dva{w}\!=\!\ena{(1,\ldots,n)\!-\!w}$ and let $\dva{w}_k$ be the complex $\dva{w}$ dimensionally shifted by $k$. Let $\mathcal{S}_n\!:=\!\big\{(w_1,\ldots,w_n)\!\in\!\{1,\ldots,n\}^n;\, w_1\!+\!\ldots\!+\!w_n\!=\!\binom{n+1}{2}\big\}$, so that $\Lambda^{\!\ast}\mathfrak{nil}_n\!=\!\bigoplus_{w\in\mathcal{S}_n}\dva{w}$. Notice that $\dva{w_1,\ldots,w_{n\!-\!1},n} \!=\! \dva{w_1,\ldots,w_{n\!-\!1}}$ and $\dva{1,w_2,\ldots,w_n} \!=\! \dva{w_2\!-\!1,\ldots,w_n\!-\!1}$.

\begin{Exp} Let us take a look at bracket subcomplexes in $\Lambda^{\!\ast}\frak{nil}_n$ for $n\!\leq\!4$.
\par Set $\mathcal{S}_2$ consists of permutations of $(1,2)$. Furthermore, there holds $H_k\dva{1,2}\!=\! H_k\langle\emptyset\rangle\!\cong\! \icases{\Z}{\text{if }k=0}{0}{\text{if }k\neq0}{\scriptstyle}{-0pt}{\Big}$ and $H_k\dva{2,1}\!=\! H_k\langle e_{12}\rangle\!\cong\! \icases{\Z}{\text{if }k=1}{0}{\text{if }k\neq1}{\scriptstyle}{-0pt}{\Big}$.
\par Set $\mathcal{S}_3$ consists of permutations of $(1,2,3),(2,2,2)$. Furthermore,
$H_k\dva{1,2,3}\!=\! H_k\langle\emptyset\rangle\!\cong\! \icases{\Z}{\text{if }k=0}{0}{\text{if }k\neq0}{\scriptstyle}{-0pt}{\Big}$,
$H_k\dva{1,3,2}\!=\! H_k\langle e_{23}\rangle\!\cong\! \icases{\Z}{\text{if }k=1}{0}{\text{if }k\neq1}{\scriptstyle}{-0pt}{\Big}$,
$H_k\dva{2,1,3}\!=\! H_k\langle e_{12}\rangle\!\cong\! \icases{\Z}{\text{if }k=1}{0}{\text{if }k\neq1}{\scriptstyle}{-0pt}{\Big}$,
$H_k\dva{2,3,1}\!=\! H_k\langle e_{13}e_{23}\rangle\!\cong\! \icases{\Z}{\text{if }k=2}{0}{\text{if }k\neq2}{\scriptstyle}{-0pt}{\Big}$,
$H_k\dva{3,1,2}\!=\! H_k\langle e_{12}e_{13}\rangle\!\cong\! \icases{\Z}{\text{if }k=2}{0}{\text{if }k\neq2}{\scriptstyle}{-0pt}{\Big}$,
$H_k\dva{3,2,1}\!=\! H_k\langle e_{12}e_{13}e_{23}\rangle\!\cong\! \icases{\Z}{\text{if }k=3}{0}{\text{if }k\neq3}{\scriptstyle}{-0pt}{\Big}$,
$H_k\dva{2,2,2}\!=\! H_k\langle e_{13},e_{12}e_{23}\rangle\!\simeq\! 0$.
\par Set $\mathcal{S}_4$ consists of permutations of $(1,\!1,\!4,\!4),(1,\!2,\!3,\!4),(1,\!3,\!3,\!3),(2,\!2,\!2,\!4),(2,\!2,\!3,\!3)$.
The largest complexes are $\dva{3,\!2,\!3,\!2}\!=\!\langle\!e_{12}e_{14}e_{24}, e_{13}e_{14}e_{34}, e_{12}e_{14}e_{23}e_{34}, e_{12}e_{13}e_{24}e_{34}\!\rangle$ and $\dva{2,\!3,\!2,\!3}\!=\!\langle e_{14}e_{23}, e_{13}e_{24}, e_{12}e_{23}e_{24}, e_{13}e_{23}e_{34}\rangle$, with $\dva{3,\!2,\!3,\!2} \!\cong\! \dva{2,\!3,\!2,\!3}_1$. See the final chapter for the complete computation of $H_{\!\ast}\frak{nil}_4$.\hfill$\lozenge$
\end{Exp}

\begin{Lmm}\label{ObratniKompleks} $\dva{w_1,\ldots,w_n}\cong \dva{n\!+\!1\!-\!w_n,\ldots,n\!+\!1\!-\!w_1}.$
\end{Lmm}
\begin{proof}
Define $\tau(e_{ab})\!=\!e_{n+1-b,n+1-a}$ and $\tau(\wedge_{i=1}^ke_{a_ib_i})= (-1)^{k+1}\!\wedge_{i=1}^k\!\tau(e_{a_ib_i})$. Now $e_M\!\in\!\dva{w_1,\ldots,w_n}$ implies $\tau(e_M)\!\in\!\dva{n\!+\!1\!-\!w_n,\ldots,n\!+\!1\!-\!w_1}$, because\vspace{-0mm}
$$\begin{array}{r@{\hspace{3pt}}l}
w_{n+1-i}(\tau(e_{a_1b_1}\ldots e_{a_kb_k})) &=w_{n+1-i}(e_{n+1-b_1,n+1-a_1\!}\ldots e_{n+1-b_k,n+1-a_k})\\
                                                                &=n\!+\!1-i-|\{j; a_j\!=\!i\}|+|\{j; b_j\!=\!i\}|\\
                                                                &=n\!+\!1-(i\!-\!|\{j; b_j\!=\!i\}|\!+\!|\{j; a_j\!=\!i\}|)\\
                                                                &=n\!+\!1-w_i(e_{a_1b_1}\ldots e_{a_kb_k}).
\end{array}\vspace{-0mm}$$
From $[\tau(e_{ab}),\tau(e_{cd})]\!=\!-\tau([e_{ab},e_{cd}])$, we obtain\vspace{-0mm}
$$\begin{array}{r@{\hspace{3pt}}l}
\partial\tau(e_{a_1b_1}\ldots e_{a_kb_k})
&=(-1)^{k+1}\sum_{i<j}(-1)^{i+j}[\tau e_{a_ib_i}\!,\tau e_{a_jb_j}\!]\ldots \widehat{\tau}(e_{a_ib_i})\ldots \widehat{\tau}(e_{a_jb_j})\ldots\\
&=(-1)^{k}\sum_{i<j}(-1)^{i+j}\tau[e_{a_ib_i},e_{a_jb_j}]\ldots \widehat{\tau}(e_{a_ib_i})\ldots \widehat{\tau}(e_{a_jb_j})\ldots\\
&=\tau\partial(e_{a_1b_1}\ldots e_{a_kb_k}),
\end{array}\vspace{-0mm}$$
so $\tau$ is a chain map. Since $\tau\circ\tau\!=\!\mathrm{id}$, our $\tau$ is an isomorphism of chain complexes.
\end{proof}

\begin{Lmm}\label{ZamaknjenKompleks} $\dva{w_1,w_2,\ldots,w_n}\cong \dva{w_2,\ldots,w_n,w_1}_{2w_1\!-n-1}.$
\end{Lmm}
\begin{proof}
Define a linear map $\varphi\!: \dva{w_1,w_2,\ldots,w_n}\longrightarrow \dva{w_2,\ldots,w_n,w_1}_{2w_1\!-n-1}$ by
$$\varphi\big(\!\wedge_{i=1}^{n\!-\!1}\! e_{\{i\}\times M_i}\big)= (-1)^{\Sigma M_{\!1}}e_{M_1^C\!\times\{n+1\}} \wedge_{i=2}^{n\!-\!1}e_{\{i\}\times M_i},$$
where $M_1^C\!=\!\{2,\ldots,n\}\!\setminus\!M_1$; it is convenient to have indices in the codomain go from $2$ to $n\!+\!1$ instead of from $1$ to $n$. There holds\vspace{-0mm}
$$\begin{array}{@{\hspace{1pt}}r@{\hspace{3pt}}l}
w_i(\varphi(e_M)) &=i\!-\!|\{x;e_{x,i+\!1}\!\in\!\varphi(e_M)\}|\!+\!|\{y;e_{i+\!1,y}\!\in\!\varphi(e_M)\}|\\
&=i\!-\!\big(|\{x;e_{x,i+\!1}\!\!\in\!e_M\}|\!+\!\icases{-\!1}{\!e_{1\!,i\!+\!1}\in e_{\!M}}{~0}{\!e_{1\!,i\!+\!1}\notin e_{\!M}}{\scriptstyle}{-2pt}{\big}\!\big) \!+\!\big(|\{y;e_{i+\!1,y}\!\!\in\!e_M\}|\!+\!\icases{1}{\!e_{1\!,i\!+\!1}\notin e_{\!M}}{0}{\!e_{1\!,i\!+\!1}\in e_{\!M}}{\scriptstyle}{-2pt}{\big}\!\big)\\
&=i\!+\!1\!-\!|\{x;e_{x,i+\!1}\!\!\in\!e_M\}|\!+\!|\{y;e_{i+\!1,y}\!\!\in\!e_M\}| =w_{i+1}(e_M)\text{ \;for }i\!<\!n \text{ \;\;and }\\
w_n(\varphi(e_M))&=n\!-\!|\{x;e_{x,n+\!1}\!\in\!\varphi(e_M)\}|\!+\!|\{y;e_{n+\!1,y}\!\in\!\varphi(e_M)\}|\\
&=n\!-\!(n\!-\!1\!-\!|M_1|)\!+\!0 =1\!-\!0\!+\!|M_1|=w_1(e_M).
\end{array}\vspace{-0mm}$$
\noindent Length difference of $e_{\!M}$ and $\varphi(e_{\!M}\!)$ is $(n\!-\!1\!-\!|M_1|)\!-\!|M_1| \!=\!n\!-\!1\!-\!2(w_1\!-\!1)\!=\!n\!+\!1\!-\!2w_1$. Thus $\varphi$ is a well-defined bijection. Denoting $M\!\setminus\!x\!\cup\!y := (M\!\!\setminus\!\!\{x\})\!\cup\!\{y\}$, we have
\[\begin{array}{l} 
\varphi\partial(e_{\{\!1\!\}\times M_1} e_{\!N})
=\varphi\big(\!\sum_{x\in M_{\!1}\!,\,y\in N_{\!x}\!\setminus\!M_{\!1}} \varepsilon_{\!xy} e_{\{\!1\!\}\times(\!M_{\!1}\!\setminus x\cup y)} e_{\!N\!\setminus\!\{\!(x\!,y)\!\}} +(-1)^{|M_{\!1}\!|}e_{\{1\}\!\times\!M_{\!1}}\partial e_{\!N}\!\big)\\
=\hspace{-9pt}\sum\limits_{x\in M_{\!1}\!,\,y\in N_{\!x}\!\setminus\!M_{\!1}}\hspace{-9pt} (-1)^{y-x+\Sigma M_{\!1}}\varepsilon_{\!xy} e_{\!(M_{\!1}\!\setminus x\cup y)^C\!\times\{n_{\!}+\!1\}} e_{\!N\!\setminus\!\{\!(x\!,y)\!\}} +(\!-\!1\!)^{|M_{\!1}\!|+\Sigma M_{\!1}}e_{\!M_{\!1}^C\!\times \{n_{\!}+\!1\}}\partial e_{\!N},\\[12pt]
\partial\varphi(e_{\{\!1\!\}\times M_{\!1}} e_{\!N}) = \partial\big((-1)^{\Sigma M_{\!1}}e_{\!M_{\!1}^C\!\times\{n_{\!}+\!1\}}e_{\!N}\big)\\
=\hspace{-9pt}\sum\limits_{y\in N_{\!x\!}\cap_{\!}M_{\!1}^C\!\!,\,x\notin M_{\!1}^C}\hspace{-9pt} (-1)^{\Sigma M_{\!1}}(\!-_{\!}\varepsilon'_{\!xy}) e_{\!(\!M_{\!1}^C\!\setminus y\cup x)\times\{n_{\!}+\!1\}} e_{\!N\!\setminus\!\{\!(x\!,y)\!\}} + (-1)^{\Sigma M_{\!1}+|M_{\!1}^C\!|}e_{\!M_{\!1}^C\!\times\{n_{\!}+\!1\}}\partial e_{\!N}\\
=\hspace{-11pt}\sum\limits_{x\in M_{\!1}\!,\,y\in N_{\!x}\!\setminus\!M_{\!1}}\hspace{-11pt} (-1)^{1+\Sigma M_{\!1}}\varepsilon'_{\!xy} e_{\!(\!M_{\!1}\!\setminus x\cup y)^C\!\times\{n_{\!}+\!1\}} e_{\!N\!\setminus\!\{\!(x\!,y)\!\}} + (-1)^{n_{\!}-\!1+|M_{\!1}\!|+\Sigma M_{\!1}}e_{\!M_{\!1}^C\!\times\{n_{\!}+\!1\}}\partial e_{\!N\!}.\\
\end{array}\]
for $\varepsilon_{\!xy},\varepsilon'_{\!xy}\!\in\!\{1,-\!1\}$. Since $[e_{y,n_{\!}+_{\!}1},e_{x,y}]\!=\!-e_{x,n+1}$, there is a minus before $\varepsilon'_{\!xy}$. We must show that $(-1)^{y-x}\varepsilon_{\!xy}=(-1)^{n}\varepsilon'_{\!xy}$: if $\alpha\!=\!($position of $x$ in $M_1)$, $\beta\!=\!($position of $(x,y)$ in $N)$, $\gamma\!=\!($position of $y$ in $M_1\!\setminus\!x\!\cup\!y)$, then $y$ in $M_1^C$ has position $y\!-\!\gamma\!-\!1$ and $x$ in $M_1^C\!\setminus\!y\!\cup\!x$ has position $x\!-\!\alpha$, so $(-1)^{y-x}\varepsilon_{\!xy} \!=\! (-1)^{y-x+\alpha+(|M_{\!1}|+\beta)+(\gamma-1)} \!=\! (-1)^{n+(y-\gamma-1)+(n-1-|M_{\!1}|+\beta)+(x-\alpha-1)} \!=\! (-1)^n\varepsilon'_{\!xy}$. Therefore \vspace{-4pt} $\varphi\partial\!=\!(-1)^{n\!-\!1}\partial\varphi$, hence $\overline{\varphi}(e_{\!M}\!):= \icases{\varphi(e_{\!M}\!)(\!-\!1\!)^{n\!-\!1}}{\text{if }|M|\in2\N}{\varphi(e_{\!M}\!)}{\text{if }|M|\notin2\N}{\scriptstyle}{-0pt}{\Big}$ is an isomorphism of chain complexes.
\end{proof}

\begin{Lmm}\label{IndukcijaKompleks} $\dva{w_1,\ldots,w_{k\!-\!1},n,w_{k\!+\!1},\ldots,w_n}\cong\dva{w_1,\ldots,w_{k\!-\!1},w_{k\!+\!1},\ldots,w_n}_{n-k}$ and $\dva{w_1,\ldots,w_{k\!-\!1},1,w_{k\!+\!1},\ldots,w_n}\cong\dva{w_1\!-\!1,\ldots,w_{k\!-\!1}\!-\!1,w_{k\!+\!1}\!-\!1,\ldots,w_n\!-\!1}_{k-1}$.
\end{Lmm}
\begin{proof}
We can identify $\dva{w_1,\ldots,w_{n\!-\!1},n}$ with $\dva{w_1,\ldots,w_{n\!-\!1}}$. By Lemma \ref{ZamaknjenKompleks},
\begin{align*}
\dva{w_{1\!},\ldots,\!w_{k\!-\!1},\!n,\!w_{k\!+\!1},\ldots,\!w_n\!}
&\!\cong\!\dva{w_{k\!+\!1},\ldots,\!w_n,\!w_1,\ldots,\!w_{k\!-\!1},\!n}_{-2\!\sum_{i=k\!+\!1}^n\!\!w_i+(n_{\!}-_{\!}k)(n_{\!}+\!1_{\!})}\\
&\!\cong\!\dva{w_{k\!+\!1},\ldots,\!w_n,\!w_1,\ldots,\!w_{k\!-\!1}}_{-2\!\sum_{i=k\!+\!1}^n\!\!w_i+(n_{\!}-_{\!}k)(n_{\!}+\!1_{\!})}\\
&\!\cong\!\dva{w_1,\ldots,\!w_{k\!-\!1},\!n,\!w_{k\!+\!1},\ldots,\!w_n}_{(n_{\!}-_{\!}k)\cdot1}.
\end{align*}
We can identify $\dva{1,w_2,\ldots,w_n\!}$ with $\dva{w_2\!\!-\!\!1,\ldots,w_n\!\!-\!\!1}$. By Lemma~\ref{ZamaknjenKompleks},
\begin{align*}
\dva{w_{1\!},\ldots,\!w_{k\!-\!1},\!1,\!w_{k\!+\!1},\ldots,\!w_n\!}
&\!\cong\!\dva{1,\!w_{k\!+\!1},\ldots,\!w_n,\!w_1,\ldots,\!w_{k\!-\!1}}_{2\!\sum_{i=1}^{k\!-\!1}\!\!w_i-(k\!-\!1_{\!})(n_{\!}+\!1_{\!})}\\
&\!\cong\!\dva{w_{k\!+\!1}\!\!-\!\!1,\ldots,\!w_n\!\!-\!\!1,\!w_1\!\!-\!\!1,\ldots,\!w_{k\!-\!1}\!\!-\!\!1}_{2\!\sum_{i=1}^{k\!-\!1}\!\!w_i-(k\!-\!1_{\!})(n_{\!}+\!1_{\!})}\\
&\!\cong\!\dva{w_1\!\!-\!\!1,\ldots,\!w_{k\!-\!1}\!\!-\!\!1,\!w_{k\!+\!1}\!\!-\!\!1,\ldots,\!w_n\!\!-\!\!1}_{2(k\!-\!1_{\!})-(k\!-\!1_{\!})}.
\end{align*}
This establishes the first and second part of the claim.
\end{proof}

\par If all elements in $w\!=\!(w_1,\ldots,w_n)$ are distinct, then by applying Lemma \ref{IndukcijaKompleks} $n$ times, we see that $\dva{w}$ has only one generator, namely \vspace{-1mm}
\begin{equation}\label{eq.2.1}
\textstyle{e_{\!\pi}:=\bigwedge_{i<j,w_i>w_j}\!e_{ij}},\vspace{-1mm}
\end{equation}
which is the wedge of inversions of permutation $w\!=\!\pi$ of $(1,\ldots,\!n)$. Indeed, $w_k(e_{\!\pi}\!) = k\!-\!|\{i;e_{ik}\!\in\!e_{\!\pi}\}|\!+\!|\{j;e_{kj}\!\in\!e_{\!\pi}\}| = k\!-\!|\{i;i\!<\!k, w_i\!>\!w_k\}|\!+\!|\{j;k\!<\!j, w_k\!>\!w_j\}| = 1\!+\!|\{i;i\!<\!k, w_i\!<\!w_k\}|\!+\!|\{j;k\!<\!j, w_k\!>\!w_j\}| = 1\!+\!|\{r;r\!\neq\!k,w_r\!<\!w_k\}| = w_k$.
\par Let $\mathcal{F}_n\!=\!\{(w_1,\ldots,w_n)\!\in\!\mathcal{S}_n;\, w_i\!\neq\!w_j\text{ for }i\!\neq\!j\}\!=\!\{\pi(1,\ldots,n);\pi\!\in\!S_n\}$. In Lemma \ref{BrezProstegaDela}, we show that for $w\!\notin\!\mathcal{F}_n$ the homology of $\dva{w}$ has only torsion. Therefore the free part is $FH_\ast(\mathfrak{nil}_n)= \bigoplus_{w\in\mathcal{F}_{\!n}} \dva{w}$. Thus we calculate what was already known to Kostant \cite{citearticleKostantLAHGBWT,citearticleHanlonLM}, 
who used the Laplacian method (the fact that $\dim H_k(C_\ast,\partial_\ast)=\dim\Ker(\partial_{k_{\!}+_{\!}1}\partial_{k_{\!}+_{\!}1}^t\!\!+\!\partial_k^t\!\partial_k)$ over $\Q$) to obtain part (a) of the following result:

\begin{Thm}\label{ProstiDel} \textbf{(a)} $F H_k(\mathfrak{nil}_n)\cong \Z^{T(n,k)}$\!\!, where $T(n,k)$ is the Mahonian number.\\
\textbf{(b)} $H^{\ast\!}(\frak{nil}_n;\Q)\cong \langle x_{\!\pi};\,\pi\!\in\!S_n\rangle \leq\Lambda_\Q[x_{ij};1\!\leq\!i\!<\!j\!\leq\!n]$, where $x_\pi\!=\!\bigwedge_{i<j,\pi_i>\pi_j}\!x_{ij}$.
\end{Thm}
In words, the cohomology algebra over $\Q$ is isomorphic to the subalgebra of the polynomial exterior algebra, spanned by the inversions of all permutations.
\begin{proof}
\textbf{(a)} By definition (OEIS A008302), $T(n,k)$ is the number of permutations of $[n]$ which have $k$ inversions, so the result follows from (\ref{eq.2.1}) and Lemma \ref{BrezProstegaDela}.\\ 
\textbf{(b)} For any $\alpha\!\in\!H^i(\frak{g})$ and $\beta\!\in\!H^j(\frak{g})$, the cup product is given by \vspace{-3pt}
$$(\alpha\!\smile\!\beta)(x_1\!\cdots\!x_{i+j})= \!\!\sum_{\pi\in S_{i+j},\,\pi_1<\ldots<\pi_i, \pi_{i+1}<\ldots<\pi_{i+j}}\!\!\! \mathrm{sgn}\pi\,\alpha(x_{\pi_1}\!\!\cdots\!x_{\pi_i})\, \beta(x_{\pi_{i+1}}\!\!\cdots\!x_{\pi_{i+j}}).\vspace{-3pt}$$
In our case, $H^{\ast\!}(\frak{nil}_n;\Q)$ is spanned by the duals $\{x_\pi\!:=\!e_\pi^\ast;\, \pi\!\in\!S_n\}$. Furthermore, $(e_{\pi}^\ast\!\smile\!e_{\pi'}^\ast)(e_{\!M})= \icases{1}{\text{if }e_{\!M}=e_{\!\pi}e_{\!\pi'}}{0}{\text{if }e_{\!M}\neq e_{\!\pi}e_{\!\pi'}}{\scriptstyle}{-0pt}{\Big}$, hence $e_{\pi}^\ast\!\smile\!e_{\pi'}^\ast= (e_{\!\pi}\!\wedge\!e_{\!\pi'}\!)^\ast$.
\end{proof}\clearpage

\vspace{3mm}
\section{Filtrations}
Let $w\!=\!(2,w_2,\ldots,w_n)$, so every wedge in $\dva{w}$ contains exactly one $e_{1\ast}$. There is a natural filtration of $\dva{w}$ by subcomplexes: if $F^w_k$ is spanned by $\{e_{1i}e_M; i\!\geq\!k\}$, then $0\!=\!F^w_{n+1}\!\leq\! F^w_n\!\leq\!\ldots\!\leq\!F^w_2\!=\!\dva{w}$. The quotient $F^w_k\!/\!F^w_{k+1}$ has generators $\{[e_{1k} e_M];\, e_{1k}e_M\!\in\!F^w_k\}$ and boundary $\partial[e_{1k}e_{\!M}]\!=\!-[e_{1k}\,\partial e_{\!M}]$, therefore
\begin{equation}\label{eq.3.1}\begin{array}{r@{\:\cong\:}l}
F^w_k\!/\!F^w_{k+1} &\dva{1,w_2,\ldots,w_{k\!-\!1},w_{k_{\!}}\!+_{\!}\!1,w_{k_{\!}+_{\!}1},\ldots,w_n}_1\\
                    &\dva{w_{2\!}\!-\!\!1,\ldots,w_{k\!-\!1\!}\!-\!\!1,w_k,w_{k_{\!}+_{\!}1\!}\!-\!\!1,\ldots,w_n\!\!-\!\!1}_1. \\
\end{array}\end{equation}

\begin{Lmm}\label{TriDvojke} If $w_r\!=\!w_s\!=\!w_t\!\in\!\{2,n_{\!}\!-_{\!}\!1\}$ for distinct $r,s,t$, then $H_{\!\ast}\dva{w_1,\ldots,\!w_n}\cong0$.
\end{Lmm}
\begin{proof}
By Lemmas \ref{ZamaknjenKompleks} and \ref{ObratniKompleks}, we may assume that $r\!=\!1$ and $w_r\!=\!w_s\!=\!w_t\!=\!2$. For any $i\!\notin\!\{s,t\}$ there holds $F^w_i\!/\!F^w_{i+1} \cong \dva{w_{2\!}\!-\!\!1,\ldots,1,\ldots,w_i,\ldots,1,\ldots,w_n\!\!-\!\!1}_1 \cong \dva{w_{2\!}\!-\!\!2,\ldots,0,\ldots,w_i\!\!-\!\!1,\ldots,w_n\!\!-\!\!2}_t \cong0$, by (\ref{eq.3.1}) and Lemma \ref{IndukcijaKompleks}.
Thus we have $0\!=\!F^w_n\!=\!\ldots\!=\!F^w_{t+1}\!<\!F^w_t\!=\!\ldots\!=\!F^w_{s+1}\!<\!F^w_s\!=\!\ldots\!=\!\dva{w}$ and a long exact sequence of a pair $\ldots \!\to\! H_{k_{\!}+_{\!}1}\frac{\dva{w}}{F^w_t} \!\overset{\chi}{\to}\! H_kF^w_t\!\to\!H_k\dva{w} \!\to\! H_k\frac{\dva{w}}{F^w_t} \!\overset{\chi}{\to}\! H_{k_{\!}-\!1}F^w_t \!\to\! \ldots$. To prove $H_\ast\dva{w}\!\cong\!0$, it suffices to show that $\chi$ is an isomorphism, where $\chi(x\!+\!F^w_t)\!=\![\partial(x)]$.
\par Let $x\!\in\!\dva{w}\!/\!F^w_t\!=\!F^w_s\!/\!F^w_{s_{\!}+\!1}$, so $x\!=\!e_{1s}\cdots$. By $w_t\!=\!2$, $x\!=\!e_{1s}e_{\{2,\ldots,s,\ldots,t\!-\!1\}\!\times\!\{t\}}\cdots$. By $w_s\!=\!2$, $x\!=\![e_{1s} e_{\{2,\ldots,s\!-\!1\}\!\times\!\{s\}} e_{\{2,\ldots,s,\ldots,t\!-\!1\}\!\times\!\{t\}} e_M]$ with no indices $s$ and $t$ in $M$.
\par Let $y\!\in\!F^w_t\!=\!F^w_t\!/\!F^w_{t_{\!}+\!1}$, so $y\!=\!e_{1t}\cdots$. Since $w_s\!=\!2$, $y\!=\!e_{1t}e_{\{2,\ldots,s\!-\!1\}\!\times\!\{s\}}\cdots$. Since $w_t\!=\!2$, $y\!=\!e_{1t} e_{\{2,\ldots,s\!-\!1\}\!\times\!\{s\}} e_{\{2,\ldots,\hat s,\ldots,t\!-\!1\}\!\times\!\{t\}} e_M$ with no indices $s$ and $t$ in $M$.
\par Since $H_k\frac{\dva{w}}{F^w_t}\!=\!\frac{\mathrm{Ker}\partial}{\mathrm{Im}\partial}$, its elements are sent by $\partial$ to $F^w_t$, so in $x$ the only multiplication is $[e_{1s},e_{st}]\!=\!e_{1t}$. Thus $\chi$ sends $x\!\mapsto\!y$ and is bijective.
\end{proof}

\begin{Lmm}\label{ZaporedniDvojki} $\dva{\ldots,2,2,\ldots}\simeq0$ and $\dva{\ldots,n\!\!-\!\!1,n\!\!-\!\!1,\ldots}\simeq0$.
\end{Lmm}
\begin{proof}
By Lemmas \ref{ZamaknjenKompleks} and \ref{ObratniKompleks}, it suffices to show that $\dva{w}\!:=\!\dva{2,2,w_3,\ldots,w_n}\simeq0$. Now $\dva{w}$ consists of $e_{1i} e_M$ with $i\!\geq\!3$ and $e_{12}e_{2i}e_M$, where 2 is not an index in $M$. Hence $\mathcal{M}\!=\!\{e_{12}e_{2i}e_{\!M}\!\to\! e_{1i}e_{\!M};\, e_{1i}e_{\!M}\!\in\!\dva{w}\}$ is a Morse matching with $\mathring{\mathcal{M}}\!=\!\emptyset$.
\end{proof}

\begin{Lmm}\label{ZamaknjenKompleksPriDva} Let $w\!=\!(2,w_2,w_3,\ldots,w_n)$ and $w'\!=\!(2,w_3,\ldots,w_n,w_2)$. Then $F_3^w\cong \dva{w'}_{2w_2\!-n-2}/F_n^{w'}$\!. If $H_\ast\dva{w_2,w_{3\!}\!-\!\!1,\ldots,w_{n\!}\!-\!\!1}\cong0$, then $H_\ast\dva{w} \cong H_\ast\dva{w'}_{2w_2-n-2}$.
\end{Lmm}
\begin{proof} Define a linear map $\varphi\!: F_3^w\to \dva{w'}_{2w_2-n-2}/F_n^{w'}$ by
$$\varphi\bigl(e_{1b}\wedge_{i=2}^{n\!-\!1}\! e_{\{i\}\!\times\!M_i}\bigr)= (-1)^{\Sigma M_2}[e_{1b}\,e_{\!M_2^C\!\times\!\{n_{\!}+_{\!}1\}}\!\wedge_{i=3}^{n\!-\!1}\! e_{\{i\}\!\times\!M_i}],$$
where $M_2^C\!=\!\{3,\ldots,n\}\!\setminus\!M_2$ and indices in the codomain are $1,3,\ldots,n\!+\!1$. Our $\varphi$ is a bijection and proof that it is a chain map is similar to the one in Lemma \ref{ZamaknjenKompleks}.
\par Let $H_{\!\ast}\dva{w_2,w_3\!-\!\!1,\ldots,w_n\!-\!\!1}\cong0$, which by (\ref{eq.3.1}) is $H_{\!\ast}(F^w_2\!/\!F^w_3)$. By the long exact sequence and first part, $H_{\!\ast}\dva{w} \!=\! H_{\!\ast}F^w_2 \!\cong\! H_{\!\ast}F^w_3 \!\cong\! H_{\!\ast}(\dva{w'}_{2w_2\!-n-2}/\!F_n^{w'})$\!. By (\ref{eq.3.1}), $H_{\!\ast} F^{w'}_n \cong H_{\!\ast}\dva{w_3\!-\!1,\ldots,w_n\!-\!1,w_2} \cong H_{\!\ast}\dva{w_2,w_3\!-\!1,\ldots,w_n\!-\!1}_{n-2w_2} \cong0$, so by the long exact sequence, $H_{\!\ast}(\dva{w'}/F^{w'}_n)\cong H_{\!\ast}\dva{w'}$ and the result follows.
\end{proof}

Recall that any chain map $\varphi\!:B_\ast\!\to\!C_\ast$ induces a chain complex $D_\ast\!=\!\cone\varphi$, where $D_n\!=\!B_{n\!-\!1}\!\oplus\!C_n$ and $\partial(b,c)\!=\!\big(\partial(b),\varphi(b)\!-\!\partial(c)\big)$. Furthermore, there is an exact sequence $\ldots\!\to\!H_{k_{\!}+\!1}D_\ast \!\to\! H_kB_\ast \!\overset{\varphi_\ast}{\to}\! H_kC_\ast \!\to\! H_kD_\ast \!\to\! H_{k_{\!}-\!1}B_\ast\!\overset{\varphi_\ast}{\to}\!\ldots$.
\begin{Lmm}\label{DveTrojki}
Let $w\!=\!(2,w_2,\ldots,w_k,3,3,w_{k_{\!}+_{\!}3},\ldots,w_n)$ and $w'\!=\!(w_2\!\!-\!\!2,\ldots,w_k\!\!-\!\!2,$ $3,w_{k_{\!}+_{\!}3}\!-\!2,\ldots,w_n\!-\!2)$. Then $H_{\!\ast}\dva{w}\cong H_{\!\ast}\!\cone\!\big(\dva{w'}_k \!\!\overset{\cdot2}{\longrightarrow}\!\! \dva{w'}_k\!\big)$.
\end{Lmm}
\begin{proof}
Let $k\!=\!\!1$, so $w\!=\!(2,\!3,\!3,\ldots)$. By (\ref{eq.3.1}) and Lemma \ref{ZaporedniDvojki}, $F^w_k\!/\!F^w_{k+1}\!\simeq\!0$ for $k\!\geq\!4$, so $H_{\!\ast}F^w_4\!\cong\!0$ and  $H_{\!\ast}\dva{w}\cong H_{\!\ast}(\dva{w}/\!F^w_4)$. There are 4 types of generators in $\dva{w}/\!F^w_4$:
\begin{itemize}[leftmargin=5mm]
\item $A=\{[e_{12}e_{23}e_{2a}e_{3b} e_M];\, \text{all indices in $M$ are $\geq\!4$}\}$,
\item $B=\{[e_{13}e_{23}e_{3a}e_{3b} e_M];\, \text{all indices in $M$ are $\geq\!4$}\}$,
\item $C=\{[e_{12}e_{2a}e_{2b}e_M];\;\;\;\;\;\text{all indices in $M$ are $\geq\!4$}\}$,
\item $D=\{[e_{13}e_{2a}e_{3b}e_M];\;\;\;\;\;\text{all indices in $M$ are $\geq\!4$}\}$.
\end{itemize}
The set $\mathcal{M}=\{ A\!\ni\!e_{12}e_{23}e_{2a}e_M\!\to\! e_{13}e_{2a}e_M\!\in\!D\}$ is a Morse matching, with critical elements $\mathring{\mathcal{M}}\!=\!B\!\cup\!C$. Nontrivial zig-zag paths go from $B$ to $C$ and come in pairs:\vspace{-1mm}
\[\hspace{-0pt}\begin{tikzpicture}[description/.style={fill=white, inner sep=1pt}]
     \matrix (a) [matrix of math nodes, row sep=2mm, column sep=9mm, text height=1.5ex, text depth=1.25ex]
    {{[e_{13}e_{23}e_{3a}e_{3b}e_{\!M}]\!\!} & {\!\![e_{13}e_{2a}e_{3b}e_{\!M}]}\\
     {[e_{12}e_{23}e_{2a}e_{3b}e_{\!M}]\!\!} & {\!\![e_{12}e_{2a}e_{2b}e_{\!M}]}\\};
     \path[->] (a-1-1.east) edge node[description]{\tiny$1$}   (a-1-2.west);
     \path[->] (a-2-1.east) edge node[description]{\tiny$-\!1$}(a-1-2.west);
     \path[->] (a-2-1.east) edge node[description]{\tiny$1$}   (a-2-2.west);
\end{tikzpicture} \;\raisebox{11pt}{and}\;
\begin{tikzpicture}[description/.style={fill=white, inner sep=1pt}]
     \matrix (a) [matrix of math nodes, row sep=2mm, column sep=9mm, text height=1.5ex, text depth=1.25ex]
    {{[e_{13}e_{23}e_{3a}e_{3b}e_{\!M}]\!\!} & {\!\![e_{13}e_{2b}e_{3a}e_{\!M}]}\\
     {[e_{12}e_{23}e_{2b}e_{3a}e_{\!M}]\!\!} & {\!\![e_{12}e_{2a}e_{2b}e_{\!M}]},\\};
     \path[->] (a-1-1.east) edge node[description]{\tiny$-\!1$} (a-1-2.west);
     \path[->] (a-2-1.east) edge node[description]{\tiny$-\!1$} (a-1-2.west);
     \path[->] (a-2-1.east) edge node[description]{\tiny$-\!1$} (a-2-2.west);
\end{tikzpicture}\vspace{-2mm}\]
which add up to $\cdot2$. We have $\langle \mathring{\mathcal{M}}\rangle\!/\!\langle C\rangle \!\cong\! \dva{w'}_2$ (omit $e_{13}e_{23}$ and indices $1,2$) and $\langle C\rangle \!\cong\! \dva{w'}_1$ (omit $e_{12}$ and indices $1,3$), so $H_{\!\ast}\dva{w}\!\cong\! H_{\!\ast} \langle\mathring{\mathcal{M}}\rangle\!\cong\! H_{\!\ast}\!\cone\bigl(\dva{w'}_1\!\stackrel{\cdot2}{\to}\!\dva{w'}_1\!\bigr)$.
\par Finally, if $k\!\geq\!2$, then $H_{\!\ast}\dva{w} \cong H_{\!\ast}\dva{2,3,3,\!w_{k+3},\ldots,\!w_n,\!w_2,\ldots,\!w_k}_{\sum_{i=2}^k\!(2w_i-n-2)}$ $\cong H_{\!\ast}\cone\big( {\scriptstyle\cdot2}\,\rotatebox[origin=c]{-90}{$\circlearrowright$}\, \dva{3,\!w_{k+3}\!-\!2,\ldots,\!w_n\!-\!2,\!w_2\!-\!2,\ldots,\!w_k\!-\!2}_{1+\sum_{i=2}^k\!(2w_i-n-2)}\big) \cong $
$H_{\!\ast}\cone$ $\big( {\scriptstyle\cdot2}\,\rotatebox[origin=c]{-90}{$\circlearrowright$}\, \dva{w_2\!-\!2,\ldots,w_k\!-\!2,3,w_{k+3}\!-\!2,\ldots,w_n\!-\!2}_{ 1+\sum_{i=2}^k\!(2w_i-n-2)-\sum_{i=2}^k\!(2(w_i-2)-(n-1))}\big) \cong H_{\!\ast}\!\cone\!\big(\dva{w'}_k \!\overset{\cdot2}{\longrightarrow}\! \dva{w'}_k\!\big)$ by Lemmas \ref{ZaporedniDvojki}, \ref{ZamaknjenKompleksPriDva}, \ref{ZamaknjenKompleks}, so the job is done.
\end{proof}

\begin{Lmm}\label{StozecMedDvojkama} Let $w\!=\!(2,w_2,\ldots,w_{k_{\!}-\!1},2,w_{k_{\!}+\!1},\ldots,w_n)$.
\begin{enumerate}[leftmargin=6mm]
\item Let $A\!=\!\{e_{\{1,\ldots,k_{\!}-\!1\}\!\times\!\{k\}}e_{ka}e_M\!\in\!\dva{w}\}$ and $B\!=\!\{e_{1a} e_{\{2,\ldots,k_{\!}-\!1\}\!\times\!\{k\}}e_M\!\in\!\dva{w};\, a\!>\!k\}$. There exists a Morse matching $\mathcal{M}$ for $\dva{w}$, such that $\mathring{\mathcal{M}}\!=\!A\!\cup\!B$, $\mathring{\partial}|_B\!=\!\partial|_B$, \vspace{-1mm}
    \[\begin{array}{r@{\hspace{3pt}}l}
    \hspace{-0pt}\mathring{\partial}|_A\!\!:e_{\{1,\ldots,k_{\!}-\!1\}\!\times\!\{k\}}e_{ka} e_{\!M}
    &\mapsto (\!-\!1\!)^{k_{\!}+_{\!}1}\! \big(e_{\{1,\ldots,k_{\!}-\!1\}\!\times\!\{k\}}\partial(e_{ka}e_{\!M}\!) \!+\!n_Me_{1a}e_{\{2,\ldots,k_{\!}-\!1\}\!\times\!\{k\}}e_{\!M}\!\big)\\
    &+\sum_{(b,c)\in X} (\!-\!1\!)^{\epsilon_{bc}\!+k+1} e_{1c}e_{\{2,\ldots,k_{\!}-\!1\}\!\times\!\{k\}} e_{ba}e_{\!M\!\setminus\!\{\!(b,c)\!\}},
    \end{array}\vspace{-1mm}\]
    where $n_M\!=\!|\{b\!\in\!\{1,\ldots,k\!-\!1\};\, (b,a)\!\notin\!M\}|$, $\epsilon_{bc}\!=\!($position of $(b,c)$ in $M)$, and $X=\{(b,c)\!\in\!M;\, b\!<\!k\!<\!c,(b,a)\!\notin\!M\}$.
\item $H_{\!\ast}\dva{w} \cong H_{\!\ast}\cone\varphi$ for some chain map $\varphi\!: F^w_{k_{\!}+\!1}\!\!\to\! F^w_{k_{\!}+\!1}$.
\item $H_{\!\ast}\dva{2,\!w_2,\ldots,\!w_{n_{\!}-\!2},\!2,\!w_n}\cong H_{\!\ast}\cone\bigl({\scriptstyle\cdot(w_n\!-\!1)}\,\rotatebox[origin=c]{-90}{$\circlearrowright$}\, \dva{w_2\!-\!\!1,\ldots,\!w_{n_{\!}-\!2}\!-\!\!1,\!1,\!w_n}_1\bigr)$.
\end{enumerate}
\end{Lmm}
\begin{proof}
\emph{(1):} There are four types of generators in $\dva{w}$: $A$, $B$,
\begin{itemize}[leftmargin=5mm]
\item[] $C=\{e_{1a} e_{\{2,\ldots,k_{\!}-\!1\}\!\times\!\{k\}} e_M;\: a\!<\!k,\text{ there is no index 1 or $k$ in }M\}$,
\item[] $D=\{e_{\{1,\ldots,\hat a,\ldots,k_{\!}-\!1\}\!\times\!\{k\}}e_M;\: 1\!<\!a\!<\!k,\text{ there is no index 1 or $k$ in }M\}$.
\end{itemize}
Set $\mathcal{M}\!=\!\big\{C\!\ni\!e_{1a} e_{\{2,\ldots,k_{\!}-\!1\}\!\times\!\{k\}}e_M \!\to\! e_{\{1,\ldots,\hat{a},\ldots,k_{\!}-\!1\}\!\times\!\{k\}}e_M\!\in\!D\big\}$ is a Morse matching, with $\mathring{\mathcal{M}}\!=\!A\!\cup\!B$. Zig-zag paths starting in $B$ are arrows and end in $B$, so $\mathring{\partial}|_B\!=\!\partial|_B$. Zig-zag paths starting in $A$ are $e_{\{1,\ldots,k_{\!}-\!1\}\!\times\!\{k\}}e_{ka}e_M \raisebox{-3pt}{$\overset{(\!-\!1\!)^{k\!+\!1}}{\longrightarrow}$}\! e_{1a}e_{\{2,\ldots,k_{\!}-\!1\}\!\times\!\{k\}}e_M$ and \vspace{-1mm} 
\[\begin{tikzpicture}[description/.style={fill=white, inner sep=1pt}]
    \matrix (a) [matrix of math nodes, row sep=2mm, column sep=20mm, text height=1.5ex, text depth=1.25ex]
    {e_{\{1,\ldots,k_{\!}-\!1\}\!\times\!\{k\}} e_{ka} e_M      & e_{\{1,\ldots,\hat{b},\ldots,k_{\!}-\!1\}\!\times\!\{k\}} e_{ba} e_M \\[-2mm]
    e_{1b} e_{\{2,\ldots,k_{\!}-\!1\}\!\times\!\{k\}} e_{ba}e_M & e_{1a} e_{\{2,\ldots,k_{\!}-\!1\}\!\times\!\{k\}} e_M\text{\; for }n_M\text{ choices},\\[-2mm]
                                                                & e_{1c} e_{\{2,\ldots,k_{\!}-\!1\}\!\times\!\{k\}} e_{ba} e_{M\!\setminus\!\{\!(b,c)\!\}}\text{\; for }(b,c)\!\in\!X.\\};
    \path[->] (a-1-1.east) edge node[description]{\tiny $(\!-\!1\!)^{b}$} (a-1-2.west);
    \path[->] (a-2-1.east) edge node[description]{\tiny $(\!-\!1\!)^{b_{\!}+_{\!}1}$} (a-1-2.west);
    \path[->] (a-2-1.east) edge node[description]{\tiny $(\!-\!1\!)^{k_{\!}+_{\!}1}$} (a-2-2.west);
    \path[->] (a-2-1.east) edge node[description]{\tiny $(\!-\!1\!)^{\epsilon_{bc}\!+_{\!}k_{\!}+_{\!}1}$} (a-3-2.west);
\end{tikzpicture}\vspace{-2mm}\]
\emph{(2):} Follows from \emph{(1)}, because $\langle B\rangle\!=\!F^w_{k+1}$ and $\langle\mathring{\mathcal{M}}\rangle\!/\!\langle B\rangle\!\cong\! \langle B\rangle_1$ (we mod out $B$, so 2nd and 3rd summand in $\mathring{\partial}|_A$ are 0, thus $e_{\{1,\ldots,k_{\!}-\!1\}\!\times\!\{k\}}e_{ka}e_M \!\mapsto\! e_{1a}e_{\{2,\ldots,k_{\!}-\!1\}\!\times\!\{k\}}e_M$ is a chain isomorphism). Ergo $\varphi$ is the part of $\mathring{\partial}|_A$ that goes to $B$.\\
\emph{(3):} Follows from \emph{(2)}, since $F^w_{k+1}\!=\!F^w_n\!\cong\!\dva{w'}_1$ (omit $e_{1n}$ and index $1$) and $X\!=\!\emptyset$ and $n_M\!=\!|\{b\!\in\!\{1,\ldots,n\!-\!2\};\,(b,n)\!\notin\!M\}| \!=\!n\!-\!1\!-\!(n\!-\!w_n) \!=\!w_n\!-\!1$.
\end{proof}

\vspace{1mm} Dwyer \cite{citearticleDwyerHIUTM} reports how Kunkel proved that $H_{\!\ast}\mathfrak{nil}_n$ has $p$-torsion for prime $p\!<\!n\!-\!1$. Now we can easily see that $H_{\!\ast}\mathfrak{nil}_n$ also has $p^m$-torsion for every $p^m\!<\!n\!-\!1$:

\begin{Exp}
Let $q\!=\!p^m\!=\!n\!-\!2$ and $w\!=\!(2,3,\ldots,q\!+\!1,2,q\!+\!1)$. By Lemma \ref{StozecMedDvojkama}, $H_\ast\dva{w}\cong H_\ast\cone(\dva{w'}_1 \!\overset{\cdot q}{\longrightarrow}\! \dva{w'}_1)$. Since $w'\!=\!(2,\ldots,q,1,q\!+\!1)$ is a permutation of $(1,\ldots,q\!+\!1)$ and $|\{(i,j);\, i\!<\!j,w'_i\!>\!w'_j\}| \!=\!q\!-\!1$, we have $H_k\dva{w'}\cong \icases{\Z}{\text{if }k=q-1}{0}{\text{if }k\neq q-1}{\scriptstyle}{-0pt}{\Big}$, so $H_k\dva{w}\cong \icases{\Z_q}{\text{if }k=q}{0}{\text{if }k\neq q}{\scriptstyle}{-0pt}{\Big}$. If $q\!<\!n\!-\!2$, then $H_\ast\dva{w,q\!+\!3,\ldots,n}\cong H_\ast\dva{w}$. \hfill$\lozenge$
\end{Exp}

\vspace{1mm} In \cite{citearticleDwyerHIUTM}, Dwyer proved that there is no $p$-torsion in $H_{\!\ast}\mathfrak{nil}_n$ for any prime $p\!\geq\!n\!-\!1$. The next example shows that $H_{\!\ast}\mathfrak{nil}_n$ can have $p^m$-torsion for some $p^m\!\geq\!n\!-\!1$:

\begin{Exp}
Let $w\!=\!(2,\!4,\!7,\!5,\!4,\!2,\!5,\!7)$. Complex $\dva{w}$ is spanned by 192 wedges. Using Lemma \ref{StozecMedDvojkama}, $H_{\!\ast}\dva{w}\!\cong\! H_{\!\ast}\!\cone({\scriptstyle\varphi}\,\rotatebox[origin=c]{-90}{$\circlearrowright$}\, F^w_7)$. By (\ref{eq.3.1}), $F^w_8\!\cong\! \dva{3,\!6,\!4,\!3,\!1,\!4,\!7}_1\!\cong\! \dva{3,\!6,\!4,\!3,\!1,\!4}_1\!\cong\! \dva{2,\!5,\!3,\!2,\!3}_5\!\cong\! \dva{2,\!3,\!2,\!3}_8$, so $H_{\!10}F^w_8\!\cong\!\Z_2$ generated by $[e_{18}e_{\!M}\!=:\!a]$ for $M\!=\!\{(2,\!7),(4,\!5)\}\!\cup\!\{(i,\!6); i\!=\!2,\ldots,5\}\!\cup\!\{(3,\!i); i\!=\!4,5,7\}$. By Lemmas \ref{IndukcijaKompleks}, \ref{ObratniKompleks}, \ref{StozecMedDvojkama},
\[\begin{array}{r@{\hspace{3pt}}l}
H_{\!\ast}F^w_7\!/\!F^w_8  &\cong H_{\!\ast}\dva{3,\!6,\!4,\!3,\!1,\!5,\!6}_1\cong H_{\!\ast}\dva{2,\!5,\!3,\!2,\!4,\!5}_5\cong \\
&\cong H_{\!\ast}\dva{2,\!3,\!5,\!4,\!2,\!5}_5\cong H_{\!\ast}\!\cone\bigl({\scriptstyle\cdot4}\,\rotatebox[origin=c]{-90}{$\circlearrowright$}\, \dva{2,\!4,\!3,\!1,\!5}_6\bigr).
\end{array}\]
Because $H_{\!\ast}\dva{2,\!4,\!3,\!1,\!5}\!\cong\!\Z$ generated by $[e_{14}e_{23}e_{24}e_{34}]$, we have $H_{\!\ast}\dva{2,\!3,\!5,\!4,\!2,\!5}\!\cong\!\Z_4$ generated by $[e_{16}e_{\!N}]$ where $N\!=\!\{(2,\!3),(2,\!4),(2,\!5),(3,\!4)\}$. 
Then $H_5\dva{2,\!5,\!3,\!2,\!4,\!5}$ is generated by $[e_{16}e_{\!N'}]$ for $N'\!=\!\{(7\!-\!y,7\!-\!x); (x,y)\!\in\!N\}$, and $H_{10} F^w_7\!/\!F^w_8$\! is generated by $[e_{17}e_{28}e_{\!P}\!=:\!b]$ for $P\!=\!\{(3,\!4)_{\!},\!(3,\!5)_{\!},\!(3,\!7)_{\!},\!(4,\!5)\}\!\cup\!\{(i,\!6); i\!=\!2,\ldots,5\}$.
\par Exact sequence $\ldots\!\to\!H_{\!k_{\!}+_{\!}1}\!\frac{F^w_7}{F^w_8} \!\overset{\chi}{\to}\!H_{\!k}F^w_8  \!\overset{\iota_\ast}{\to}\!H_{\!k}F^w_7 \!\overset{\pi_\ast}{\to}\!H_{\!k}\frac{F^w_7}{F^w_8} \!\overset{\chi}{\to}\!H_{\!k_{\!}-\!1}\!F^w_8 \!\to\!\ldots$ implies $H_k F^w_7\!\cong\!0$ for $k\!\ne\!10$ and $H_{10} F^w_7\!=\!($extension of $\Z_2$ by $\Z_4)\!\cong\!(\Z_2\!\times\!\Z_4$ or $\Z_8)$ generated by $[a]$ and $[b]$. Since $\partial\bigl(\sum_{i\in\{3,4,5,7\}} e_{17}e_{2i}e_{i8}e_{\!P}\bigr) = 4e_{17}e_{28}e_{\!P}\!+\!e_{18}e_{\!M}$ $(\ast)$ in $F^w_7$, we have $0\!=\!4[b]\!+\![a]$ in $H_{10}F^w_7$\!, hence $[b]$ itself is a generator and $H_{10}F^w_7\!\cong\!\Z_8$. 
\par Let us compute $\varphi_\ast$. In the proof of Lemma \ref{StozecMedDvojkama} for our case, $a,b\!\in\!\langle B\rangle$ and $x\!:=\!e_{16}e_{67}e_{28}e_{\!P}\!=\! e_{\{1,\ldots,5\}\!\times\!\{6\}}e_{67}e_{\!R} \!\in\!\langle A\rangle\!\cong\!\langle B\rangle_1$ for $R\!=\!\{(2,\!8)_{\!},\! (3,\!4)_{\!},\! (3,\!5)_{\!},\! (3,\!7)_{\!},\! (4,\!5)\}$;
$$\varphi(x)=\mathring{\partial}(x) =\pm n_{\!R}e_{17}e_{\{2,\ldots,5\}\!\times\!\{6\}}e_{\!R} \pm\!e_{18}e_{\{2,\ldots,5\}\!\times\!\{6\}}e_{27}e_{\!R\!\setminus\!\{\!(2,8)\!\}}= \pm4b \pm\!a.$$
By $(\ast)$, $\varphi_\ast(x)$ is 0 or $\pm8[b]$. In both cases, $\varphi_\ast\!=\!0$. Exact sequence $\ldots\!\to\!H_{\!k_{\!}+\!1}\dva{w} \!\to\! H_{\!k}F^w_7 \!\overset{\varphi_\ast}{\to}\! H_{\!k}F^w_7 \!\to\! H_{\!k}\dva{w} \!\to\! H_{\!k_{\!}-\!1}F^w_7\!\overset{\varphi_\ast}{\to}\!\ldots$ then implies $H_{\!k}\dva{w}\cong \icases{\Z_8}{\text{if }k\in\{10,11\}}{0}{\text{if }k\notin\{10,11\}}{\scriptstyle}{-0pt}{\Big}$. \hfill$\lozenge$
\end{Exp}

\vspace{3mm}\vfill
\section{Free part of homology}
We can generalise the filtration from the previous section to an arbitrary complex $\dva{w_1,\ldots,w_n}\!=\!\dva{w}$. In every $e_{\!M}\!\in\!\dva{w}$, there are exactly $t\!:=\!w_1\!-\!1$ occurrences of $e_{1\ast}$. Thus $F^w_k\!:=\!\big\langle e_{1i_1}\ldots e_{1i_t}e_{\!M};\, i_1\!\!+\!\ldots\!+\!i_t\!\geq\!k,\, 1\text{ is not in }M\big\rangle \!\leq\!\dva{w}$ is a subcomplex.
\par Define $w(i_1,\ldots,i_t)\!\in\!\mathcal{S}_{n_{\!}-\!1}$ as $(w'_2,\ldots,w'_n)$, where $w'_j\!=\!\icases{w_j}{\text{if }j\in\{i_1\!,\ldots,i_t\!\}}{w_j\!-\!1}{\text{if }j\notin\{i_1\!,\ldots,i_t\!\}}{\scriptstyle}{-1pt}{\big}$. Then \vspace{-1mm}
\begin{equation}\label{eq.4.1} \textstyle{F^w_k\!/\!F^w_{k+1}\cong \bigoplus_{i_1\!+\ldots+i_t=k}\dva{w(i_1,\ldots,i_t)}_t.}\end{equation}\clearpage

\begin{Exp}\label{Primer33333}
Let us compute $H_{\!\ast}\dva{3,\!3,\!3,\!3,\!3}$. By Lemmas \ref{ZaporedniDvojki}, \ref{StozecMedDvojkama}, (\ref{eq.4.1}), $F^w_{10}\!=\!0$,\, $F^w_9\!/\!F^w_{10} \!\cong\!\dva{2,\!2,\!3,\!3}_2\!\simeq\!0$,\, $H_{\!\ast}F^w_8\!/\!F^w_9\!\cong\! H_{\!\ast}\dva{2,\!3,\!2,\!3}_2\!\cong\!\Z_2$ generated by $[e_{13}e_{15}\,e_{24}e_{35}]$,\, $F^w_7\!/\!F^w_8\!\cong\! \dva{3,\!2,\!2,\!3}_2\!\oplus\!\dva{2,\!3,\!3,\!2}_2\!\simeq\!0$,\, $H_{\!\ast}F^w_6\!/\!F^w_7\!\cong\! H_{\!\ast}\dva{3,\!2,\!3,\!2}_2\!\cong\! H_{\!\ast}\dva{2,\!3,\!2,\!3}_3\cong\Z_2$ generated by $[e_{12}e_{14}\,e_{23}e_{25}e_{35}]$,\,  $F^w_5\!/\!F^w_6\!\cong\! \dva{3,\!3,\!2,\!2}_2\!\simeq\!0$,\, $F^w_5\!=\!\dva{w}$. Thus $H_{\!\ast}\dva{w} \!\cong\! H_{\!\ast}F_6$ and $H_{\!\ast}F^w_7 \!\cong\! H_{\!\ast}F^w_8 \!\cong\! H_{\!\ast}F^w_8\!/\!F^w_9 \!\cong\! \Z_2$ generated by $[e_{13}e_{15}\,e_{24}e_{35}]$. In the exact sequence $\ldots\!\to\!H_{\!k_{\!}+\!1}\!\frac{F^w_6}{F^w_7} \!\overset{\chi}{\to}\!H_{\!k}F^w_7 \!\to\!H_{\!k}F^w_6 \!\to\!H_{\!k}\frac{F^w_6}{F^w_7} \!\overset{\chi}{\to}\!H_{\!k_{\!}-\!1}\!F^w_7\!\to\!\ldots$ our $\chi$ sends \vspace{-1mm}
\[\begin{array}{r @{\hspace{3pt}} l}
[e_{12}e_{14}\,e_{23}e_{25}e_{35}]&\mapsto [e_{13}e_{14}e_{25}e_{35}\!+\!e_{14}e_{15}e_{23}e_{35}]=\\
&\,[-e_{13}e_{15}e_{24}e_{35}\!-\!\partial(e_{13}e_{15}e_{23}e_{34}e_{35})\!+\!\partial(e_{13}e_{14}e_{25}e_{34}e_{45})].
\end{array}\]
It is an isomorphism, hence by exactness, $H_{\!\ast}\dva{3,\!3,\!3,\!3,\!3}\cong H_{\!\ast}F^w_6\cong0$. \hfill$\lozenge$
\end{Exp}

\begin{Lmm}\label{BrezProstegaDela} For $w\!\in\!\mathcal{S}_n\!\!\!\setminus\!\!\mathcal{F}_{\!n}\!=:\!\mathcal{T}_{\!n}$, the homology of $\dva{w}$ has only torsion.
\end{Lmm}
\begin{proof}
We use induction on $n$. Computing over $\Q$, by universal coefficient theorem, it suffices to show that $H_{\!\ast}\dva{w}\!\cong\!0$. The claim is trivial for $n\!=\!2$ since $\mathcal{T}_2\!=\!\emptyset$. Let $n\!>\!2$ and $w\!\in\!\mathcal{T}_n$. By Lemma \ref{IndukcijaKompleks} and induction hypothesis, we may assume that $2\!\leq\!w_1,\ldots,w_n\!\le\!n\!-\!1$. By Lemma \ref{ZamaknjenKompleks} we may assume that $w_1\!\leq\!w_i$ for all $i$. If $w_1\!\geq\!3$, all elements of the sequence $w(i_1,\ldots,i_{w_1\!-\!1})$ are at least 2, so $w(i_1,\ldots,i_{w_1\!-\!1})\!\in\! \mathcal{T}_{n\!-\!1}$. By (\ref{eq.4.1}) and induction, $H_{\!\ast} F^w_k\!/\!F^w_{k_{\!}+\!1}\cong0$ for all $k$, hence $H_{\!\ast}\dva{w}\cong0$.
\par Let $w_1\!=\!2$. If $w(i)\!\in\! \mathcal{T}_{n\!-\!1}$ for all $i$, then by (\ref{eq.4.1}) and induction, $H_{\!\ast} F^w_i\!/\!F^w_{i_{\!}+\!1}\cong0$ for all $i$, hence $H_{\!\ast}\dva{w}\cong0$. Suppose there exists $i\!\in\!\{2,\ldots,n\}$ with $w(i)\!\notin\!\mathcal{T}_{n\!-\!1}$. Then $w_i\!=\!n\!-\!1$, $w_1\!=\!2$, and other elements of $w$ form a permutation of $(2,\ldots,n\!-\!1)$, so there are exactly two numbers $j\!<\!i$ such that $w(j),w(i)\!\in\!\mathcal{F}_{n\!-\!1}$. In the filtration $0\!=\!F^w_{n+1} \!\!\leq\! \ldots\!\leq\!F^w_{i+1} \!\!\leq\! F^w_i \!\!\leq\! \ldots \!\leq\! F^w_{j+1} \!\!\leq\! F^w_j \!\!\leq\! \ldots \!\leq\! F^w_2\!\!=\!\dva{w}$, by (\ref{eq.3.1}) and (\ref{eq.2.1}) we have $H_{\!k} F^w_i\!/\!F^w_{i+1}\!\cong\! \icases{\Q}{\text{if }k=1+|I_i|}{0}{\text{if }k\neq1+|I_i|}{\scriptstyle}{-0pt}{\Big}$, $H_{\!k} F^w_j\!/\!F^w_{j+1}\!\cong\! \icases{\Q}{\text{if }k=1+|I_j|}{0}{\text{if }k\neq1+|I_j|}{\scriptstyle}{-0pt}{\Big}$, and $H_{\!\ast} F^w_r\!/\!F^w_{r+1}\!\cong\!0$ for $r\!\notin\!\{i,j\}$, where $I_i\!=\!\{$inversions of $w(i)\}$ and $I_j\!=\!\{$inversions of $w(j)\}$. Therefore $0\!\cong\!H_{\!\ast}F_{n+1}\!\cong\!\ldots\!\cong\!H_{\!\ast}F_{i+1}$, $H_{\!\ast}F_i\!\cong\! \icases{\Q}{\text{if }k=1+|I_i|}{0}{\text{if }k\neq1+|I_i|}{\scriptstyle}{-0pt}{\Big}$, $H_{\!\ast} F^w_i\!\cong\!\ldots\!\cong\!H_{\!\ast} F^w_{j+1}$, $H_{\!\ast} F^w_j\!\cong\!\ldots\!\cong\!H_{\!\ast}\dva{w}$. Consequently, it suffices to show that in the exact sequence\vspace{-1mm}
\[\small\begin{tikzpicture}[description/.style={fill=white, inner sep=1pt}]
    \matrix (a) [matrix of math nodes, row sep=4mm, column sep=3mm, text height=1.5ex, text depth=1.25ex]
    {\!\!\!\!\ldots\!\! & \!H_{\!k_{\!}+\!2}\!\frac{F^w_j}{F^w_{j_{\!}+\!1}}\! &\!H_{\!k_{\!}+\!1}\!F^w_{j_{\!}+\!1}\! &\!H_{\!k_{\!}+\!1}\!F^w_j\! &\!H_{\!k_{\!}+\!1}\!\frac{F^w_j}{F^w_{j_{\!}+\!1}}\! &\!H_{\!k}\!F^w_{j_{\!}+\!1}\! &\!H_{\!k}\!F^w_j\! &\!H_{\!k}\!\frac{F^w_j}{F^w_{j_{\!}+\!1}}\! &\!H_{\!k_{\!}-\!1}\!F^w_{j_{\!}+\!1}\! &\!\!\ldots\\[-2mm]
            &0                            &0                &             &\Q                           &\Q           &         &0                        &0\\};
    \path[->](a-1-1.east)edge node{}(a-1-2.west); \path[->](a-1-2.east)edge node[above]{$\chi$}(a-1-3.west); \path[->](a-1-3.east)edge node{}(a-1-4.west);
    \path[->](a-1-4.east)edge node{}(a-1-5.west); \path[->](a-1-5.east)edge node[above]{$\chi$}(a-1-6.west); \path[->](a-1-6.east)edge node{}(a-1-7.west);
    \path[->](a-1-7.east)edge node{}(a-1-8.west); \path[->](a-1-8.east)edge node[above]{$\chi$}(a-1-9.west); \path[->](a-1-9.east)edge node{}(a-1-10.west);
    \draw[white](a-1-2.south)edge node[rotate=90,black]{$\cong$}(a-2-2.north); \draw[white](a-1-3.south)edge node[rotate=90,black]{$\cong$}(a-2-3.north);
    \draw[white](a-1-5.south)edge node[rotate=90,black]{$\cong$}(a-2-5.north); \draw[white](a-1-6.south)edge node[rotate=90,black]{$\cong$}(a-2-6.north);
    \draw[white](a-1-8.south)edge node[rotate=90,black]{$\cong$}(a-2-8.north); \draw[white](a-1-9.south)edge node[rotate=90,black]{$\cong$}(a-2-9.north);
\end{tikzpicture}\vspace{-4mm}\]
the morphism $\chi_{k+1}$ for $k\!=\!|I_j|$ is bijective. Notice that $H_{\!k_{\!}+\!1}\!\frac{F^w_j}{F^w_{j_{\!}+\!1}}$ and $H_{\!k}\!F^w_{j_{\!}+\!1}\!\cong\! H_{\!k}\!\frac{F^w_i}{F^w_{i_{\!}+\!1}}$ are generated by $[e_{1j}\bigwedge_{(a,b)\in I_j}\!e_{a+1,b+1}]\!=:\![x]$ and $[e_{1i}\bigwedge_{(a,b)\in I_i}\!e_{a+1,b+1}]\!=:\![y]$. Since $I_j\!=\!I_{i\!}\cup_{\!}\{(j\!-\!1,i\!-\!1)\}$, we get $\chi([x])\!=\![\partial(x)]\!=\![\pm y]$, which concludes the argument.
\end{proof}

\begin{Lmm}\label{CorDual} $H_k\dva{w_n,\ldots, w_1}\cong H_{\binom{n}{2}-k-1}\dva{w_1,\ldots,w_n}$ for $(w_1,\ldots,w_n)\!\in\!\mathcal{T}_{\!n}$.
\end{Lmm}
\begin{proof}
Let $(C^\ast\!,\delta^\ast)$ be the dual of the complex $(C_\ast,\partial_\ast)$ of $\frak{nil}_n$, let $f_{\!M}$ be the dual of a basis element $e_{\!M}$, and $N\!=\!\binom{n}{2}$. Define $\tau_\ast\!: C_\ast \!\to\! C^{N\!-\ast}$  by $\tau(e_{\!M})\!=\!\varepsilon_{\!M}f_{\!M^C}$, where $\varepsilon_{\!M}$ is the sign of the permutation $(M\!,M^C)$ of $\{(i,j);\, 1\!\le\!i\!<\!j\!\le\!n\}$. By \cite[p.640]{citearticleHazewinkelDTCLA}, $\tau_\ast$ is a chain isomorphism, i.e. $\tau_{k\!-\!1}\partial_k= \delta^{N\!-k}\tau_k$. For $e_{\!M}\!\in\!\dva{w_1,\ldots,w_n}$ we have
\[\begin{array}{r@{\hspace{3pt}}l}
w_i\big(\tau(e_{\!M})\big)&=i\!-\!|\{x;(x,i)\!\in\!M^C\}|\!+\!|\{x;(i,x)\!\in\!M^C\}|\\
&=i\!-\!\bigl(i\!-\!1\!-\!|\{x;(x,i)\!\in\!M\}|\bigr)\!+\!\bigl(n\!-\!i\!-\!|\{x;(i,x)\!\in\!M\}|\bigr)= n\!+\!1\!-\!w_i,
\end{array}\]
hence $\tau(\dva{w_1,\ldots,w_n}) \!=\! \dva{n\!+\!1\!-\!w_1,\ldots,n\!+\!1\!-\!w_n}^\ast \!\cong\! \dva{w_n,\ldots,w_1}^\ast$ by Lemma \ref{ObratniKompleks}. Now the result follows from Lemma \ref{BrezProstegaDela} and the universal coefficient theorem.
\end{proof}

\vspace{3mm}
\section{Computations}
We have $H_{\!\ast}\mathfrak{nil}_n\cong \bigl(\bigoplus_{w\in\mathcal{F}_{\!n}}\!H_{\!\ast}\dva{w}\bigr)\oplus \bigl(\bigoplus_{w\in\mathcal{T}_{\!n}}\!H_{\!\ast}\dva{w}\bigr)$. Free part $\bigoplus_{w\in\mathcal{F}_{\!n}}H_{\!\ast}\dva{w}$ is known from Theorem \ref{ProstiDel}. For the torsion part, we use Lemma \ref{IndukcijaKompleks}:\vspace{-1.5mm} 
\[\begin{array}{r@{\hspace{3pt}}l}
TH_\ast(\mathfrak{nil}_n)&\cong \bigl(\oplus_{1,n\notin w\in \mathcal{T}_{\!n}}H_\ast\dva{w}\bigr)\oplus\frac{\bigl(\oplus_{1\in w\in \mathcal{T}_{\!n}}H_\ast\dva{w}\bigr)
\oplus \bigl(\oplus_{n\in w\in \mathcal{T}_{\!n}}H_\ast\dva{w}\bigr)}{\oplus_{1,n\in w\in \mathcal{T}_{\!n}}H_\ast\dva{w}}\\[4pt]
&\cong \bigl(\oplus_{1,n\notin w\in \mathcal{T}_{\!n}}H_\ast\dva{w}\bigr)\oplus\frac{ \oplus_{w\in \mathcal{T}_{\!n\!-\!1}} \oplus_{k=1}^{n} H_\ast\dva{w}_{k_{\!}-_{\!}1}\oplus
 H_\ast\dva{w}_{n_{\!}-_{\!}k}}{ \oplus_{w\in\mathcal{T}_{\!n\!-\!2}} \oplus_{i\in[n\!-\!1],j\in[n]} H_\ast\dva{w}_{i\!-\!1\!+\!n\!-\!j}}\\[4pt]
&\cong \bigl(\oplus_{1,n\notin w\in \mathcal{T}_{\!n}}H_\ast\dva{w}\bigr)\oplus\frac{\oplus_{k=0}^{n_{\!}-_{\!}1} TH_{\ast+k}(\mathfrak{nil}_{n\!-\!1})^2}
{\oplus_{i=0}^{2n_{\!}-_{\!}3} TH_{\ast+i}(\mathfrak{nil}_{n\!-\!2})^{\min\{i\!+\!1,2n\!-\!2\!-\!i\}}}.
\end{array}\vspace{-1.5mm}\]
By induction and Lemmas \ref{TriDvojke}, \ref{ZaporedniDvojki}, it suffices to calculate only $H_{\!\ast}\dva{w}$ coming from $\widetilde{\mathcal{T}}_n\!:=\!\big\{w\!\in\!\mathcal{T}_n;\: 1,n\!\notin\!w,\: \nexists i\!: w_i\!=\!w_{i+1}\!\in\!\{2,n\!-\!1\},\: \nexists i\!<\!j\!<\!k\!: w_i\!=\!w_j\!=\!w_k\!\in\!\{2,n_{\!}\!-_{\!}\!1\}\big\}$. Define maps $\alpha,\beta,\gamma\!: \widetilde{\mathcal{T}}_n \!\to\! \widetilde{\mathcal{T}}_n$ by $\alpha(w_1,\ldots,w_n)\!=\! (w_2,\ldots,w_n,w_1)$, $\beta(w_1,\ldots,w_n)\!=\! (n\!+\!1\!-\!w_n,\ldots,n\!+\!1\!-\!w_1)$, $\gamma(w_1,\ldots,w_n)\!=\! (w_n,\ldots,w_1)$. Let $\sim$ be the smallest equivalence relation on $\widetilde{\mathcal{T}}_n$ with $w\!\sim\!\xi(w)$ for $\xi\!\in\!\{\alpha,\beta,\gamma\}$. By Lemmas \ref{ZamaknjenKompleks}, \ref{ObratniKompleks}, \ref{CorDual}, we need to compute $H_{\!\ast}$ only for one complex in each equivalence class.
\par \underline{Case $n\!=\!4$}: The set $\mathcal{S}_4$ consists of all permutations of $(1,\!1,\!4,\!4)$, $(1,\!2,\!3,\!4)$, $(1,\!3,\!3,\!3)$, $(2,\!2,\!2,\!4)$, $(2,\!2,\!3,\!3)$, whilst $\widetilde{\mathcal{T}}_4/\!_\sim \!=\! \{(2,3,2,3)\}$. Now $\dva{3,2,3,2}\!\cong\!\dva{2,3,2,3}_1$ has only 4 generators, so $H_\ast$ can be computed directly, but let us use Lemma \ref{StozecMedDvojkama}: $H_\ast\dva{2,3,2,3}\cong H_{\!\ast}\cone(\dva{2,\!1,\!3}_1\!\overset{\cdot2}{\rightarrow}\!\dva{2,\!1,\!3}_1\!)$. Since $\dva{2,1,3}\!=\!\langle e_{12}\rangle$, we conclude that $H_k\dva{2,3,2,3}\!\cong\!
\icases{\Z_2}{\text{if }k=2}{0}{\text{if }k\neq2}{\scriptstyle}{-0pt}{\Big}$. Hence the torsion part is $TH_k(\mathfrak{nil}_4)\cong \icases{\Z_2}{\text{if }k\in\{2,3\}}{0}{\text{if }k\notin\{2,3\}}{\scriptstyle}{-0pt}{\Big}$.
\par \underline{Case $n\!=\!5$}: Set $\widetilde{\mathcal{T}}_5/\!_\sim$ consists of $a\!=\!(2,3,4,2,4)$, $b\!=\!(2,3,3,3,4)$, $c\!=\!(2,3,3,4,3)$, $d\!=\!(3,3,3,3,3)$. By Lemma \ref{StozecMedDvojkama}, $H_\ast\dva{a}\cong H_\ast\cone(\dva{2,3,1,4}_1\!\overset{\cdot 3}{\to}\! \dva{2,3,1,4}_1)\cong \icases{\Z_3}{\text{if }k=3}{0}{\text{if }k\neq3}{\scriptstyle}{-0pt}{\Big}$. By Lemma \ref{DveTrojki}, $H_\ast\dva{b}\cong H_\ast\cone(\dva{1,3,2}_2 \!\overset{\cdot 2}{\to}\! \dva{1,3,2}_2)\cong \icases{\Z_2}{\text{if }k=3}{0}{\text{if }k\neq3}{\scriptstyle}{-0pt}{\Big}$, and $H_\ast\dva{c}\cong H_\ast\cone(\dva{3,2,1}_1 \!\overset{\cdot 2}{\to}\! \dva{3,2,1}_1)\cong \icases{\Z_2}{\text{if }k=4}{0}{\text{if }k\neq4}{\scriptstyle}{-0pt}{\Big}$. By Example \ref{Primer33333}, $H_\ast\dva{d}\cong0$. Because $\beta(a)\!=\!(2,4,2,3,4)\!=\!\alpha^3(a)$, $\beta(b)\!=\!(2,3,3,3,4)\!=\!b$, $\beta(c)\!=\!(3,2,3,3,4)\!=\!\alpha^4(c)$, and $\gamma(x)\!\neq\! \alpha^i(x)$ for all $x\!\in\!\{a,b,c\}$ and all $i$, we conclude that\vspace{-2mm}
\[\bigoplus_{w\in \widetilde{\mathcal{T}}_{\!n}}H_k\dva{w}=
\bigoplus_{x\in\{a,b,c\}}\bigoplus_{i\in\{0,\ldots,4\}}\bigl(H_k\dva{\alpha^i(x)}\oplus H_k\dva{\gamma\alpha^i(x)}\bigr)=
\left\{\begin{smallmatrix}
\Z_2^4\oplus \Z_3^2; &k=3,\\
\Z_2^6\oplus \Z_3^3; &k=4,\\
\Z_2^6\oplus \Z_3^3; &k=5,\\
\Z_2^4\oplus \Z_3^2; &k=6.
\end{smallmatrix}\right.\vspace{-2mm}\]
\par \underline{Case $n\!=\!6$} is still doable by hand. Set $\widetilde{\mathcal{T}}_6/\!_\sim$ has 28 elements: 9 cases are done by Lemma \ref{DveTrojki}, 6 by Lemma \ref{StozecMedDvojkama}, and the rest by examining their filtration. There are only 3 classes containing no 2 or $n\!-\!1$: $(3,3,3,4,4,4)$, $(3,3,4,3,4,4)$, $(3,4,3,4,3,4)$.
\par \underline{Cases $n\!=\!7,8$} require a computer. The set $\widetilde{\mathcal{T}}_7/\!_\sim$ has 250 elements, and $\widetilde{\mathcal{T}}_8/\!_\sim$ has 3485 elements. See the table below for the homology of $\mathfrak{nil}_7$ and $\mathfrak{nil}_8$.
\par \underline{Cases $n\!\geq\!9$}: The set $\widetilde{\mathcal{T}}_9/\!_\sim$ has $59\,102$ elements. We have not been able to compute, among other things, the homology of the complex $\dva{5,\!5,\!5,\!5,\!5,\!5,\!5,\!5,\!5}$.

\vspace{3mm}
\section{Afterword}

\subsection{Conclusion} We have seen that methods, designed for a specific family of Lie algebras, where we partition the problem into smaller pieces and solve only the nontrivial nonequivalent parts, can enable us to compute more than twice as much data compared with the usual approach.
\subsection{Acknowledgment} This research was supported by the Slovenian Research Agency (research core funding no. P1-0292, J1-7025, J1-8131).

\begin{center}\begin{tabular}{ l | l l }
    $k\backslash n$ & 7 & 8\\ \hline
0 & $\Z$ & $\Z$\\
1 & $\Z^6$ & $\Z^7$ \\
2 & $\Z^{20}\!\oplus\!\Z_2^{4}$ & $\Z^{27}\!\oplus\!\Z_2^{5}$ \\
3 & $\Z^{49}\!\oplus\!\Z_2^{35}\!\oplus\!\Z_3^6$ & $\Z^{76}\!\oplus\!\Z_2^{57}\!\oplus\!\Z_3^8$ \\
4 & $\Z^{98}\!\oplus\!\Z_2^{124}\!\oplus\!\Z_3^{27}\!\oplus\!\Z_4^{6}$ & $\Z^{174}\!\oplus\!\Z_2^{253}\!\oplus\!\Z_3^{45}\!\oplus\!\Z_4^{9}$ \\
5 & $\Z^{169}\!\oplus\!\Z_2^{303}\!\oplus\!\Z_3^{78}\!\oplus\!\Z_4^{28}\!\oplus\!\Z_5^{4}$ & $\Z^{343}\!\oplus\!\Z_2^{793}\!\oplus\!\Z_3^{168}\!\oplus\!\Z_4^{53}\!\oplus\!\Z_5^{8}$ \\
6 & $\Z^{259}\!\oplus\!\Z_2^{635}\!\oplus\!\Z_3^{168}\!\oplus\!\Z_4^{65}\!\oplus\!\Z_5^{17}$ & $\Z^{602}\!\oplus\!\Z_2^{2132}\!\oplus\!\Z_3^{479}\!\oplus\!\Z_4^{164}\!\oplus\!\Z_5^{47}$ \\
7 & $\Z^{359}\!\oplus\!\Z_2^{1122}\!\oplus\!\Z_3^{275}\!\oplus\!\Z_4^{112}\!\oplus\!\Z_5^{38}$ & $\Z^{961}\!\oplus\!\Z_2^{4880}\!\oplus\!\Z_3^{1050}\!\oplus\!\Z_4^{380}\!\oplus\!\Z_5^{145}$ \\
8 & $\Z^{455}\!\oplus\!\Z_2^{1674}\!\oplus\!\Z_3^{384}\!\oplus\!\Z_4^{160}\!\oplus\!\Z_5^{56}$
  & $\Z^{1415}\!\oplus\!\Z_2^{9882}\!\oplus\!\Z_3^{1927}\!\oplus\!\Z_4^{730}\!\oplus\!\Z_5^{309}\!\oplus\!\Z_8$ \\
9 & $\Z^{531}\!\oplus\!\Z_2^{2096}\!\oplus\!\Z_3^{481}\!\oplus\!\Z_4^{196}\!\oplus\!\Z_5^{63}$
  & $\Z^{1940}\!\oplus\!\Z_2^{17721}\!\oplus\!\Z_3^{3178}\!\oplus\!\Z_4^{1200}\!\oplus\!\Z_5^{524}\!\oplus\!\Z_8^{5}$ \\
10& $\Z^{573}\!\oplus\!\Z_2^{2238}\!\oplus\!\Z_3^{522}\!\oplus\!\Z_4^{210}\!\oplus\!\Z_5^{64}$
  & $\Z^{2493}\!\oplus\!\Z_2^{27826}\!\oplus\!\Z_3^{4781}\!\oplus\!\Z_4^{1728}\!\oplus\!\Z_5^{766}\!\oplus\!\Z_8^{12}$ \\
11& $\Z^{573}\!\oplus\!\Z_2^{2096}\!\oplus\!\Z_3^{481}\!\oplus\!\Z_4^{196}\!\oplus\!\Z_5^{63}$
  & $\Z^{3017}\!\oplus\!\Z_2^{38810}\!\oplus\!\Z_3^{6504}\!\oplus\!\Z_4^{2253}\!\oplus\!\Z_5^{1007}\!\oplus\!\Z_8^{18}$ \\
12& $\Z^{531}\!\oplus\!\Z_2^{1674}\!\oplus\!\Z_3^{384}\!\oplus\!\Z_4^{160}\!\oplus\!\Z_5^{56}$
  & $\Z^{3450}\!\oplus\!\Z_2^{48576}\!\oplus\!\Z_3^{7902}\!\oplus\!\Z_4^{2720}\!\oplus\!\Z_5^{1219}\!\oplus\!\Z_8^{17}$ \\
13& $\Z^{455}\!\oplus\!\Z_2^{1122}\!\oplus\!\Z_3^{275}\!\oplus\!\Z_4^{112}\!\oplus\!\Z_5^{38}$
  & $\Z^{3736}\!\oplus\!\Z_2^{54457}\!\oplus\!\Z_3^{8614}\!\oplus\!\Z_4^{3011}\!\oplus\!\Z_5^{1351}\!\oplus\!\Z_8^{11}$ \\
14& $\Z^{359}\!\oplus\!\Z_2^{635}\!\oplus\!\Z_3^{168}\!\oplus\!\Z_4^{65}\!\oplus\!\Z_5^{17}$
  & $\Z^{3836}\!\oplus\!\Z_2^{54457}\!\oplus\!\Z_3^{8614}\!\oplus\!\Z_4^{3011}\!\oplus\!\Z_5^{1351}\!\oplus\!\Z_8^{11}$ \\
15& $\Z^{259}\!\oplus\!\Z_2^{303}\!\oplus\!\Z_3^{78}\!\oplus\!\Z_4^{28}\!\oplus\!\Z_5^{4}$
  & $\Z^{3736}\!\oplus\!\Z_2^{48576}\!\oplus\!\Z_3^{7902}\!\oplus\!\Z_4^{2720}\!\oplus\!\Z_5^{1219}\!\oplus\!\Z_8^{17}$ \\
16 & $\Z^{169}\!\oplus\!\Z_2^{124}\!\oplus\!\Z_3^{27}\!\oplus\!\Z_4^{6}$ & $\Z^{3450}\!\oplus\!\Z_2^{38810}\!\oplus\!\Z_3^{6504}\!\oplus\!\Z_4^{2253}\!\oplus\!\Z_5^{1007}\!\oplus\!\Z_8^{18}$ \\
17 & $\Z^{98}\!\oplus\!\Z_2^{35}\!\oplus\!\Z_3^{6}$ & $\Z^{3017}\!\oplus\!\Z_2^{27826}\!\oplus\!\Z_3^{4781}\!\oplus\!\Z_4^{1728}\!\oplus\!\Z_5^{766}\!\oplus\!\Z_8^{12}$ \\
18 & $\Z^{49}\!\oplus\!\Z_2^{4}$ & $\Z^{2493}\!\oplus\!\Z_2^{17721}\!\oplus\!\Z_3^{3178}\!\oplus\!\Z_4^{1200}\!\oplus\!\Z_5^{524}\!\oplus\!\Z_8^{5}$ \\
19 & $\Z^{20}$ & $\Z^{1940}\!\oplus\!\Z_2^{9882}\!\oplus\!\Z_3^{1927}\!\oplus\!\Z_4^{730}\!\oplus\!\Z_5^{309}\!\oplus\!\Z_8$ \\
20 & $\Z^6$ & $\Z^{1415}\!\oplus\!\Z_2^{4880}\!\oplus\!\Z_3^{1050}\!\oplus\!\Z_4^{380}\!\oplus\!\Z_5^{145}$ \\
21 & $\Z$ & $\Z^{961}\!\oplus\!\Z_2^{2132}\!\oplus\!\Z_3^{479}\!\oplus\!\Z_4^{164}\!\oplus\!\Z_5^{47}$ \\
22 & & $\Z^{602}\!\oplus\!\Z_2^{793}\!\oplus\!\Z_3^{168}\!\oplus\!\Z_4^{53}\!\oplus\!\Z_5^{8}$ \\
23 & & $\Z^{343}\!\oplus\!\Z_2^{253}\!\oplus\!\Z_3^{45}\!\oplus\!\Z_4^{9}$ \\
24 & & $\Z^{174}\!\oplus\!\Z_2^{57}\!\oplus\!\Z_3^{8}$ \\
25 & & $\Z^{76}\!\oplus\!\Z_2^{5}$ \\
26 & & $\Z^{27}$ \\
27 & & $\Z^7$ \\
28 & & $\Z$
\end{tabular}\end{center}\vspace{-1.5mm}
\end{document}